\documentclass[12pt,reqno]{amsart}
\usepackage{fullpage,  amssymb, amsthm, amsmath, amsfonts, times}

\usepackage{enumitem}
\usepackage{mathrsfs}
\usepackage{mathtools}
\usepackage{centernot}
\usepackage{xfrac}
\usepackage{eufrak}
\usepackage{stmaryrd}
\usepackage{bm}
\usepackage[all,cmtip]{xy}
\usepackage[usenames,dvipsnames]{xcolor}
\usepackage[margin=1.0 in]{geometry}
\usepackage{float}
\usepackage{hyperref}
\hypersetup{colorlinks=true, citecolor=blue, linkcolor=blue, urlcolor=blue, pdfstartview=FitH, pdfauthor=Curran, pdftitle=}

\usepackage{tikz-cd}
\usetikzlibrary{positioning,arrows}
\usetikzlibrary{calc}

\usepackage{todonotes}

\usepackage[utf8]{inputenc}

\begin{document}

\parskip0pt
\parindent10pt

\newenvironment{answer}{\color{Blue}}{\color{Black}}
\newenvironment{exercise}
{\color{Blue}\begin{exr}}{\end{exr}\color{Black}}

\theoremstyle{plain} 
\newtheorem{theorem}{Theorem}[section]
\newtheorem*{theorem*}{Theorem}
\newtheorem{prop}[theorem]{Proposition}
\newtheorem{porism}[theorem]{Porism}
\newtheorem{lemma}[theorem]{Lemma}
\newtheorem{cor}[theorem]{Corollary}
\newtheorem{conj}[theorem]{Conjecture}
\newtheorem{funfact}[theorem]{Fun Fact}
\newtheorem*{claim}{Claim}
\newtheorem{question}{Question}
\newtheorem*{conv}{Convention}

\theoremstyle{remark}
\newtheorem{exr}{Exercise}
\newtheorem*{rmk}{Remark}

\theoremstyle{definition}
\newtheorem{defn}{Definition}
\newtheorem{example}{Example}

\renewcommand{\mod}[1]{{\ifmmode\text{\rm\ (mod~$#1$)}\else\discretionary{}{}{\hbox{ }}\rm(mod~$#1$)\fi}}

\newcommand{\ns}{\mathrel{\unlhd}}
\newcommand{\tr}{\text{tr}}
\newcommand{\wt}[1]{\widetilde{#1}}
\newcommand{\wh}[1]{\widehat{#1}}
\newcommand{\cbrt}[1]{\sqrt[3]{#1}}
\newcommand{\floor}[1]{\left\lfloor#1\right\rfloor}
\newcommand{\abs}[1]{\left|#1\right|}
\newcommand{\ds}{\displaystyle}
\newcommand{\nn}{\nonumber}
\newcommand{\re}{\text{Re}}
\renewcommand{\ker}{\textup{ker }}
\renewcommand{\char}{\textup{char }}
\renewcommand{\Im}{\textup{Im }}
\renewcommand{\Re}{\textup{Re }}
\newcommand{\area}{\textup{area }}
\newcommand{\isom}
    {\ds \mathop{\longrightarrow}^{\sim}}
\renewcommand{\ni}{\noindent}
\renewcommand{\bar}{\overline}
\newcommand{\morph}[1]
    {\ds \mathop{\longrightarrow}^{#1}}

\newcommand{\Gal}{\textup{Gal}}
\newcommand{\Aut}{\textup{Aut}}
\newcommand{\Crypt}{\textup{Crypt}}
\newcommand{\disc}{\textup{disc}}
\newcommand{\sgn}{\textup{sgn}}
\newcommand{\del}{\partial}

\newcommand{\mattwo}[4]{
\begin{pmatrix} #1 & #2 \\ #3 & #4 \end{pmatrix}
}

\newcommand{\vtwo}[2]{
\begin{pmatrix} #1 \\ #2 \end{pmatrix}
}
\newcommand{\vthree}[3]{
\begin{pmatrix} #1 \\ #2 \\ #3 \end{pmatrix}
}
\newcommand{\vcol}[3]{
\begin{pmatrix} #1 \\ #2 \\ \vdots \\ #3 \end{pmatrix}
}

\newcommand*\wb[3]{%
  {\fontsize{#1}{#2}\usefont{U}{webo}{xl}{n}#3}}

\newcommand\myasterismi{%
  \par\bigskip\noindent\hfill
  \wb{10}{12}{I}\hfill\null\par\bigskip
}
\newcommand\myasterismii{%
  \par\bigskip\noindent\hfill
  \wb{15}{18}{UV}\hfill\null\par\medskip
}
\newcommand\myasterismiii{%
  \par\bigskip\noindent\hfill
  \wb{15}{18}{z}\hfill\null\par\bigskip
}

\newcommand{\one}{{\rm 1\hspace*{-0.4ex} \rule{0.1ex}{1.52ex}\hspace*{0.2ex}}}

\renewcommand{\v}{\vec{v}}
\newcommand{\w}{\vec{w}}
\newcommand{\e}{\vec{e}}
\newcommand{\m}{\vec{m}}
\renewcommand{\u}{\vec{u}}
\newcommand{\vecx}{\vec{e}_1}
\newcommand{\vecy}{\vec{e}_2}
\newcommand{\vo}{\vec{v}_1}
\newcommand{\vt}{\vec{v}_2}

\renewcommand{\o}{\omega}
\renewcommand{\a}{\alpha}
\renewcommand{\b}{\beta}
\newcommand{\g}{\gamma}
\newcommand{\sig}{\sigma}
\renewcommand{\d}{\delta}
\renewcommand{\t}{\theta}
\renewcommand{\k}{\kappa}
\newcommand{\ve}{\varepsilon}
\newcommand{\op}{\text{op}}

\newcommand{\Z}{\mathbb Z}
\newcommand{\ZN}{\Z_N}
\newcommand{\Q}{\mathbb Q}
\newcommand{\N}{\mathbb N}
\newcommand{\R}{\mathbb R}
\newcommand{\C}{\mathbb C}
\newcommand{\F}{\mathbb F}
\newcommand{\T}{\mathbb T}
\renewcommand{\H}{\mathbb H}
\newcommand{\B}{\mathcal B}
\newcommand{\p}{\mathcal P}
\renewcommand{\P}{\mathbb P}
\renewcommand{\r}{\mathcal R}
\renewcommand{\c}{\mathcal C}
\newcommand{\h}{\mathcal H}
\newcommand{\f}{\mathcal F}
\newcommand{\s}{\mathcal S}
\renewcommand{\L}{\mathcal L}
\newcommand{\lam}{\lambda}
\newcommand{\E}{\mathcal E}
\newcommand{\Ex}{\mathbb E}
\newcommand{\D}{\mathbb D}
\newcommand{\oh}{\mathcal O}
\newcommand{\n}{\mathcal N}
\newcommand{\I}{\mathcal I}
\newcommand{\G}{\mathcal G}
\newcommand{\J}{\mathcal J}

\newcommand{\diam}{\text{ diam}}
\newcommand{\vol}{\text{vol}}
\newcommand{\Int}{\text{Int}}

\newcommand{\Span}{\text{span}}
\newcommand{\AP}{\text{AP}}

\newcommand{\MoM}{\text{MoM}}

\newcommand{\0}{{\vec 0}}

\newcommand{\ignore}[1]{}

\newcommand{\poly}[1]{\textup{Poly}_{#1}}

\newcommand*\circled[1]{\tikz[baseline=(char.base)]{
            \node[shape=circle,draw,inner sep=2pt] (char) {#1};}}

\newcommand*\squared[1]{\tikz[baseline=(char.base)]{
            \node[shape=rectangle,draw,inner sep=2pt] (char) {#1};}}

\title{Lower bounds for shifted moments of the Riemann zeta function} 

\author{Michael J. Curran }
\email{Michael.Curran@maths.ox.ac.uk}
\address{Mathematical Institute, University of Oxford, Oxford, OX2 6GG, United Kingdom.}

\maketitle

\begin{abstract}
In previous work \cite{CurranCorrelations}, the author gave upper bounds for the shifted moments of the zeta function
\[
M_{\bm{\a},\bm{\b}}(T) = \int_T^{2T}  \prod_{k = 1}^m |\zeta(\tfrac{1}{2} + i (t + \a_k))|^{2 \b_k} dt
\]
introduced by Chandee \cite{Chandee}, where $\bm{\a} = \bm{\a}(T) =  (\a_1, \ldots, \a_m)$ and $\bm{\b} = (\b_1 \ldots , \b_m)$ satisfy $|\a_k| \leq T/2$ and $\b_k\geq 0$.
Assuming the Riemann hypothesis, we shall prove the corresponding lower bounds:
\[
M_{\bm{\a},\bm{\b}}(T)  \gg_{\bm{\b}} T (\log T)^{\b_1^2 + \cdots + \b_m^2} \prod_{1\leq j < k \leq m} |\zeta(1 + i(\a_j - \a_k) + 1/ \log T )|^{2\b_j \b_k}.
\]

\end{abstract}

\section{Introduction}\label{sec:Intro}
This paper is concerned with the shifted moments 
\begin{equation}\label{eqn:shiftedMoments}
M_{\bm{\a},\bm{\b}}(T) = \int_T^{2T}  \prod_{k = 1}^m |\zeta(\tfrac{1}{2} + i (t + \a_k))|^{2 \b_k} dt, 
\end{equation}
where $\bm{\a} = \bm{\a}(T) =  (\a_1, \ldots, \a_m)$ and $\bm{\b} = (\b_1 \ldots , \b_m)$ satisfy $|\a_k| \leq T/2$ and $\b_k\geq 0$.
These were first studied in general by Chandee \cite{Chandee}. In \cite{CurranCorrelations}, the author proved the upper bound 
\[
M_{\bm{\a},\bm{\b}}(T)  \ll_{\bm{\b}} T (\log T)^{\b_1^2 + \cdots + \b_m^2} \prod_{1\leq j < k \leq m} |\zeta(1 + i(\a_j - \a_k) + 1/ \log T )|^{2\b_j \b_k}
\]
assuming the Riemann hypothesis.
The goal of this paper is to obtain the corresponding lower bound, showing that these bounds are of the correct order or magnitude.

\begin{theorem}\label{thm:main}
Assume the Riemann hypothesis. If $\b_k \geq 0$ and $|\a_k| \leq T/2$ for $k = 1, \ldots , m,$ then
\[
M_{\bm{\a},\bm{\b}}(T) \gg_{\bm{\b}} T (\log T)^{\b_1^2 + \cdots + \b_m^2} \prod_{1\leq j < k \leq m} |\zeta(1 + i(\a_j - \a_k) + 1/ \log T )|^{2\b_j \b_k}
\]
for $T$ sufficiently large in terms of $\bm{\b}$.
\end{theorem}
\noindent
Note that the order of magnitude in this bound is consistent with the formula for the shifted moments predicted by the famous work of Conrey Farmer Keating Rubenstein and Snaith \cite{CFKRS}.
The reader may consult \cite{CurranCorrelations} for more background on the shifted moments and a discussion of the previous estimates for $M_{\bm{\a},\bm{\b}}(T)$.

The method of proof is similar to a principle pioneered in the works of Heap and Soundararajan \cite{HS} and Radziwiłł and Soundararajan \cite{RS-LB}.
These works demonstrate that if one can asymptotically evaluate the twisted $2k^{\text{th}}$ moment of a family of $L$-functions for some integer $k$, then one can obtain sharp lower bounds for the $2\b^\text{th}$ moments for all $\b \geq 0$.
We will need a more elaborate argument in our setting since there are more parameters for the shifted moments.
To explain the method, we will first introduce some notation.
Set
\[
\b_\ast := \sum_{k\leq m} \max(1, \b_k).
\]
We will choose a sequence of parameters $T_j = T^{c_j}$, where
\[
c_0 = 0 \text{ and } c_j = \frac{e^{j}}{(\log_2 T)^2}
\]
for $j > 0$, where throughout this paper we will use $\log_j$ to denote the $j$-fold iterated logarithm.
Let $L$ be the largest integer such that $T_L \leq T^{\d}$ where $0 < \d < e^{-1000\b_\ast}$ is some small constant depending on $\bm{\b}$ to be chosen later.
For any $X$ let
\[
\p_{1,X}(s) = \sum_{p \leq T_1} \frac{1}{p^{s+1/\log X}} \frac{\log X/p}{\log X} + \sum_{p\leq \log T} \frac{1}{2p^{2s}},
\]
and given any $2\leq j \leq L$ define
\[
\p_{j,X}(s) = \sum_{p\in (T_{j-1},T_j]} \frac{1}{p^{s+1/\log X}} \frac{\log X/p}{\log X}. 
\]
\begin{rmk}
In our proof, we could have also used the simpler Dirichlet polynomial
\[
\sum_{p\in (T_{j-1},T_j]} \frac{1}{p^s}
\]
in place of $\p_{j,X}(s)$ for all $1\leq j \leq L$. 
However the $\p_{j,X}$ will appear when using the Riemann hypothesis, and using the $\p_{j,X}$ from the onset will reduce the total number of mean value calculations.
\end{rmk}
If $\p_{j,X}(s)$ is not too large, then (see lemma \ref{lem:expTaylorSeries}) we will be able to efficiently approximate $\exp(\b \p_{j,X} (s))$ with the following Dirichlet polynomial
\[\label{eq:TaylorExp}
\n_{j,X}(s;\b) := \sum_{m \leq 100\b_\ast^2 K_j} \frac{\b^m \p_{j,X}(s)^m}{m!},
\]
where $K_j = c_j^{-3/4}$ for $j \geq 1$.
The set
\begin{equation*}
\G := \left\{t\in[T/2,5T/2]: |P_{j,T_L} (\tfrac{1}{2} + i t)| \leq K_j \text{ for all } 1\leq j\leq L\right\}.
\end{equation*}
consists of the $t\in [T/2,5T/2]$ for which $\n_{j,T_L}(s;\b)$ is a good approximation to $\exp(\b \p_{j,T_L} (s))$.
We will be computing moments over the set
\[
\G_m :=  \left\{t\in [T,2T]: t + \a_k \in \G \text{ for all } 1\leq k \leq m \right\}.
\]

Fix a smooth function $w$ such that $1_{[5/4,7/4]}(t) \leq w(t) \leq 1_{[1,2]}$, and  suppose for a moment that we could compute the integral
\begin{align*}
\I_0 &= \int_{\G_m} \prod_{k\leq m} \zeta(\tfrac{1}{2} + i(t+\a_k))\\
&\times \prod_{j\leq L} \exp\left((\b_k-1) \p_{j,T_L}(\tfrac{1}{2}+i(t+\a_k)) + \b_k  \p_{j,T_L}(\tfrac{1}{2}-i(t+\a_k)) \right) w(t/T) \ dt.
\end{align*}
Then Hölder's inequality implies
\[
|\I_0| \leq M_{\bm{\a}, \bm{\b}}(T)^{1/p} \times |\J|^{1/q} \times \prod_{k \leq m} |\I_k|^{1/r_k},
\]
where
\[
\frac{1}{p} = \frac{1}{4\b_\ast}, \quad \frac{1}{r_k} = \frac{1}{2m} - \frac{\b_k}{4m\b_\ast}, \quad \frac{1}{q} = 1 - \frac{1}{p} - \sum_{k\leq m} \frac{1}{r_k}
\]
are conjugate exponents,
\begin{align*}
\I_k = \int_{\G_m} |\zeta(\tfrac{1}{2} + i (t + \a_k))|^{2m} \prod_{j\leq L} &\exp\left(2(\b_k-m)\Re \p_{j,T_L}(\tfrac{1}{2}+i(t+\a_k)) \right) \\
&\times \prod_{\substack{\ell \leq m \\ \ell \neq k}} \prod_{j\leq L} \exp\left(2\b_\ell\Re \p_{j,T_L}(\tfrac{1}{2}+i(t+\a_\ell)) \right) w(t/T) dt,
\end{align*}
and 
\begin{align*}
\J = \int_{\G_m} \prod_{k\leq m}\prod_{j\leq L} \exp\left(2\b_k\Re \p_{j,T_L}(\tfrac{1}{2}+i(t+\a_k)) \right) w(t/T) dt.
\end{align*}
Therefore, Theorem \ref{thm:main} will follow from the following three propositions.

\begin{prop}\label{prop:JBound}
 For large $T$
\[
\J \ll_{\bm{\b}}  T (\log T)^{\b_1^2 + \cdots + \b_m^2} \prod_{1\leq j < k \leq m} |\zeta(1 + i(\a_j - \a_k) + 1/ \log T )|^{2\b_j \b_k}.
\]
\end{prop}

\begin{prop}\label{prop:I0Bound}
Assuming the Riemann hypothesis, for large $T$
\[
|\I_0| \gg_{\bm{\b}}  T (\log T)^{\b_1^2 + \cdots + \b_m^2} \prod_{1\leq j < k \leq m} |\zeta(1 + i(\a_j - \a_k) + 1/ \log T )|^{2\b_j \b_k}.
\]
\end{prop}

\begin{prop}\label{prop:IkBound}
Assuming the Riemann hypothesis, for large $T$ and $1\leq k \leq m$
\[
\I_k \ll_{\bm{\b}}  T (\log T)^{\b_1^2 + \cdots + \b_m^2} \prod_{1\leq j < k \leq m} |\zeta(1 + i(\a_j - \a_k) + 1/ \log T )|^{2\b_j \b_k}.
\]
\end{prop}

To prove Propositions \ref{prop:JBound} and \ref{prop:IkBound}, 
we will use  the definition of $\G_m$ to show that
\[
\J \ll \int_T^{2T} \prod_{k\leq m} |\n_{T_L}(\tfrac{1}{2} + i (t+\a_k);\b_k) |^2 dt
\]
and
\[
\I_k \ll \int_{\G_m} |\zeta(\tfrac{1}{2} + i (t + \a_k))|^{2m} |\n_{T_L}(\tfrac{1}{2} + i (t+\a_k); \b_k - m) |^2 \prod_{\substack{\ell \leq m \\ \ell \neq k}} |\n_{T_L}(\tfrac{1}{2} + i (t+\a_\ell);\b_\ell) |^2  dt,
\]
where 
\[
\n_X(s;\b) := \prod_{j\leq L} \n_{j,X}(s;\b).
\]
Therefore $\J$ can be controlled by computing the mean value of a Dirichlet polynomial, and the $\I_k$ can be bounded by computing twisted $2m^{\text{th}}$ moments of the zeta function.
Both of these quantities can be controlled as $\n_{T_L}(s;\b)$ is a short Dirichlet polynomial.
For the latter task, we will need to invoke the Riemann hypothesis when $m > 2$.

The proof of Proposition \ref{prop:I0Bound} is a bit more difficult. The first step is to show that $\I_0$ is approximately equal to 
\begin{align*}
\int_{\G_m} \prod_{k\leq m} \zeta(\tfrac{1}{2} + i(t+\a_k)) \n_{T_L}(\tfrac{1}{2} + i(t + \a_k); \tfrac{1}{2}(\b_k - 1))^2 \overline{\n_{T_L}(\tfrac{1}{2} + i(t + \a_k); \tfrac{1}{2}\b_k)}^2 w(t/T) \ dt
\end{align*}
using the definition of $\G_m$ and the Riemann hypothesis.
To compute this quantity, we will write 
\[
\int_{\G_m} = \int_{T}^{2T} - \int_{[T,2T] \setminus \G_m}.
\]
The integral over the entire interval $[T,2T]$ is a shifted pure $m^\text{th}$ moment of zeta twisted by a Dirichlet polynomial, which we can calculate unconditionally because $\n_{T_L}(s;\b)$ is a short Dirichlet polynomial.
To bound the integral over $[T,2T] \setminus \G_m$, we will first decompose this set according to which subsum $\p_{j,X}(\tfrac{1}{2}+ i(t + \a_\ell);\b)$ is unusually large for some $\ell$, and then bound  the contribution of such $t$ by
\begin{align*}
\frac{1}{K_j^{2r}}\int_{T}^{2T} \prod_{k\leq m} |\zeta(\tfrac{1}{2} + i(t+\a_k))| &\cdot  |\n_{T_L}(\tfrac{1}{2} + i(t + \a_k); \tfrac{1}{2}(\b_k - 1)) \n_{T_L}\tfrac{1}{2} + i(t + \a_k); \tfrac{1}{2}\b_k)|^2 \\
&\quad \times|\p_{j,X}(\tfrac{1}{2} + i (t + \a_\ell))|^{2r} \ dt
\end{align*}
for some large integer $r$. To bound this quantity we are again forced to use the Riemann hypothesis, at least when $m > 2$.
We can estimate this integral using the same method in \cite{CurranCorrelations}.

We now have all the necessary tools to begin our proof of Theorem \ref{thm:main}.
To help simplify the notation going forward, we will omit subscripts depending on $\bm{\b}$ from all big-O or Vinogradov asymptotic notation.
The reader should keep in mind that all of the implicit constants in this notation can, and usually will, depend on $\bm{\b}$.
We will also implicitly assume that $T$ is sufficiently large in terms of $\bm{\b}$.
After covering some preliminary results, we will prove Proposition \ref{prop:JBound}, followed by Proposition \ref{prop:I0Bound}, and we will conclude the proof by establishing Proposition \ref{prop:IkBound}.

\section*{Acknowledgements}
\noindent
The author would like to thank Hung M. Bui and Stephen Lester for helpful discussions.

\section{Preliminary Results}

To control the size of zeta on the half line, we will use the following lemma due to Soundararajan \cite{SRHMoments} and Harper \cite{Harper}.

\begin{lemma}\label{lem:logZetaUpperBound}
Assume RH, let $t\in [T,2T]$, and $|\a| \leq T/2$. Then for $2\leq X \leq T^2$ 
\begin{align*}
\log|\zeta(\tfrac{1}{2} &+ i (t + \a))|\leq \Re \sum_{p \leq X} \frac{1}{p^{1/2 + 1/\log X+ i(t + \a)}} \frac{\log X/p}{\log X}  
\\ + &\sum_{p\leq \min(\sqrt{X}, \log T)} \frac{1}{2 p^{1 + 2 i (t +\a)}} + \frac{\log T}{\log X} + O(1).
\end{align*}
\end{lemma}
\noindent
We will often Taylor expand exponentials of the $\p_{j,X}$ in order to turn moment generating function calculations into computing mean values of short Dirichlet polynomials– a much simpler task.
The following lemma will be used frequently to accomplish this task.
\begin{lemma}\label{lem:expTaylorSeries}
If $\b \leq \b_\ast$ and $|\p_{j,X}(s)| \leq K_j$ for some $1\leq j \leq L$, then
\[
\exp(\b \p_{j,X}(s)) = (1+ O(e^{-50\b_\ast^2 K_j}))^{-1} \n_{j,X}(s;\b).
\]
\end{lemma}
\begin{proof}
Since $|\p_{j,X}(s)| \leq 2 K_j$, Taylor expansion gives
\[
 \n_{j,X}(s;\b)  =  \exp(\b \p_{j,X}(s)) + O(e^{-100\b_\ast^2 K_j}).
\]
By assumption $\exp(-2 K_j \b_\ast) \leq |\exp(\b \p_{j,X}(s))| \leq \exp(2 K_j \b_\ast)$, so the claim follows.
\end{proof}

Once we have reduced the proof to a number of mean value calculations, we will use a few standard results to compute these averages.
The first result is due to Montgomery and Vaughan (see for example theorem 9.1 of \cite{IK}).
\begin{lemma}\label{lem:MVDP}
Given any complex numbers $a_n$
\[
\int_{T}^{2T} \left|\sum_{n\leq N} \frac{a_n}{n^{i t}}\right|^2 dt = (T + O(N)) \sum_{n\leq N} |a_n|^2.
\]
\end{lemma}
\noindent
We will also make use of the property that  Dirichlet polynomials supported on distinct sets of primes are approximately independent in the mean square sense. 
The precise formulation we will use is the following splitting lemma which appears in equation (16) of \cite{HS}.
\begin{lemma}\label{lem:Splitting}
Suppose for $1\leq j \leq J$ we have $j$ disjoint intervals $I_j$ and Dirichlet polynomials $A_j(s)= \sum_{n} a_j(n)n^{-s}$ such that $a_j(n)$ vanishes unless $n$ is composed of primes in $I_j$.
Then if $\prod_{j\leq J} A_j(s)$ is a Dirichlet polynomial of length $N$
\begin{align*}
\int_T^{2T} \prod_{j\leq J} |A_j(\tfrac{1}{2} + i t)|^2 dt= (T + O(N))\prod_{j\leq J}\left(\frac{1}{T}\int_T^{2T}|A_j(\tfrac{1}{2} + i t)|^2 dt\right)
\end{align*}
\end{lemma}
\noindent
Finally, the following result due to Soundararajan \cite[lemma 3]{SRHMoments} will simplify calculations for high moments of Dirichlet polynomials supported on primes.
\begin{lemma}\label{lem:MVDPPrimes}
Let $r$ be a natural number and suppose $N^r \leq T/\log T$. Then given any complex numbers $a_p$
\[
\int_{T}^{2T} \left|\sum_{p\leq N} \frac{a_p}{p^{i t}}\right|^{2r} dt \ll T r! \left(\sum_{p\leq N} |a_p|^2\right)^r.
\]
\end{lemma}
\noindent

To compute all the mean values that will arise, we will need two more ingredients.
The first is a special case of lemma 3.2 of \cite{Koukoulopoulos}, which will let us estimate a sum over primes that will appear frequently throughout the paper.
\begin{lemma}\label{lem:cosPrimeSum}
Given $\d \in \R$ and $X \geq 2$  
\begin{align*}
\sum_{p\leq X} &\frac{\cos(\d \log p)}{p}  = \log|\zeta(1 + 1/\log X +i\d)| + O(1).
\end{align*}
\end{lemma}
\noindent 
The second and final ingredient we will need is a good estimate for the coefficients of the Dirichlet polynomials
\[
\prod_{k\leq m}  \n_{j,X}(s+ i \a_k; \b_k) := \sum_{n} \frac{b_{j,X,\bm{\a},\bm{\b}}(n)}{n^s}. 
\]
To this end, denote $a_X(p):= \log(X/p)  p^{-1/\log X}/\log X$ and define multiplicative functions $g_X$ and $h_X$ by
\[
g_X(p^r;\b) :=  \frac{\b^r a_X(p)^{r}}{r!}
\]
and
\[
h_X(p^r;\b) :=  g_X(p^r; \b) + 1_{p\leq \log T} \sum_{t = 1}^{r/2} \frac{\b^{r-t} a_X(p)^{r-t}}{2^t t! (r-2t)!} .
\]
Next define $c_{1}(n)$ to be 1 if $n$ can be written as $n = n_1 \cdots n_r$ where $r \leq 100\b_\ast^2 K_1$ and each $n_i$ is either a prime $\leq T_1$ or a prime square $\leq \log T$. 
Finally for $2\leq j \leq L$ set $c_{j}(n)$ to be 1 if $n$ is the product of at most $100\b_\ast^2 K_j$ not necessarily distinct primes in $(T_{j-1},T_j]$.
In proposition 3.1 of \cite{CurranCorrelations} it was shown that

\begin{prop}\label{prop:Njcoeffs}
For $2\leq j \leq L$
\[
\n_{j,X}(s;\b) = \sum_{p\mid n \Rightarrow p \in (T_{j-1},T_j] } \frac{g_X(n;\b)c_{j}(n)}{n^s}.
\]
If
\[
\n_{1,X}(s;\b) = \sum_{p\mid n \Rightarrow p \in (T_{j-1},T_j]} \frac{f_X(n;\b)}{n^s}
\]
then $f_X(n;\b) \leq h_X(n;\b) c_{1}(n)$ and $f_X(p;\b) = g_X(p;\b)$.
\end{prop}
\noindent
Therefore  $b_{1,X,\bm{\a},\bm{\b}}(n)$ is the $m$-fold Dirichlet convolution of $f_X(n;\b_k) n^{-i\a_k}$ and $b_{j,X,\bm{\a},\bm{\b}}(n)$ is the $m$-fold convolution of $g_X(n;\b_k) c_{j}(n) n^{-i\a_k}$ for $2 \leq j \leq L$.
We will make use of two more sets of coefficients.
Let  $b'_{j,X,\bm{\a},\bm{\b}}(n)$ be the $m$-fold  convolution of $h_X(n;\b_k) n^{-i\a_k} 1_{p|n \Rightarrow p\in (T_0,T_1]}$ when $j=1$ and  the $m$-fold  convolution  of $g_X(n;\b_k) n^{-i\a_k}1_{p|n \Rightarrow p\in (T_{j-1},T_j]}$ when $2\leq j \leq L$.
Finally, let $b''_{1,X,\bm{\a},\bm{\b}}(n)$ be the $m$-fold  convolution of $h_X(n;\b_k) 1_{p|n \Rightarrow p\in (T_0,T_1]}$ when $j= 1$ and the $m$-fold  convolution of $g_X(n;\b_k) 1_{p|n \Rightarrow p\in (T_{j-1},T_j]}$ when  $2\leq j \leq L$.
Note $b'_{j,X,\bm{\a},\bm{\b}}$ and $b''_{j,X,\bm{\a},\bm{\b}}$ are multiplicative, and $|b_{j,X,\bm{\a},\bm{\b}}(n)| , |b'_{j,X,\bm{\a},\bm{\b}}(n)| \leq b''_{j,X,\bm{\a},\bm{\b}}(n)$.
All the information we will need about these coefficients is contained in lemma 3.2 of \cite{CurranCorrelations}.

\begin{lemma}\label{lem:nCoeffPatrol}

For $1\leq j \leq L$ and $p \in (T_{j-1},T_j]$
\[
b_{j, X, \bm{\a}, \bm{\b}}(p) = a_X(p) \sum_{k = 1}^m \b_k p^{-i\a_k},
\]
and $b''_{j, X, \bm{\a}, \bm{\b}}(p) \leq \b_\ast$. If $r \geq 2$ 
\[
b''_{j, X, \bm{\a}, \bm{\b}}(p^r) \leq \frac{\b_\ast^r m^r}{r!}
\]
holds whenever $2\leq j \leq L$ or $p > \log T$, and otherwise
\[
b''_{1, X, \bm{\a}, \bm{\b}}(p^r) \leq  m \b_\ast^r r^{2m} e^{-r \log (r/m)/2m + 2r}.
\]

\end{lemma}

\section{Proof of Proposition \ref{prop:JBound}}

We will begin with the evaluation of 
\begin{align*}
\J = \int_{\G_m} \prod_{k\leq m}\prod_{j\leq L} \exp\left(2\b_k\Re \p_{j,T_L}(\tfrac{1}{2}+i(t+\a_k)) \right) w(t/T) dt,
\end{align*}
which is the simplest computation.
To compute this mean value, we apply lemma \ref{lem:expTaylorSeries} and then extend the integral over the entire set $[T,2T]$ to find 
\begin{align*}
\J \ll \int_T^{2T} \prod_{k\leq m} |\n_{T_L}(\tfrac{1}{2}+i(t+\a_k) ; \b_k)|^2 dt.
\end{align*}
Here we have used that $\prod_{j\leq L} (1 +O(e^{-50\b_\ast^2 K_j}))^2 = O(1)$.
Note $\n_{j,X}(s;\b)$ is a Dirichlet polynomial of length $\leq T_j^{100m\b_\ast^2 K_j}$ when $j > 1$ and of length $\leq T_1^{200m\b_\ast^2 K_1}$ when $j=1$,  so 
\[
\prod_{k\leq m} \n_{X}(s + i \a_k ; \b_k)
\]
has length at most $T_1^{200m\b_\ast^2 K_1} T_2^{100m\b_\ast^2 K_2} \cdots T_L^{100m\b_\ast^2 K_L} \leq T^{1/10}$ by choice of  $T_j, K_j$ and $L$.
Therefore using lemma \ref{lem:Splitting} in tandem with proposition 3.3 of \cite{CurranCorrelations} we find
\begin{align*}
\J &\ll  T \prod_{j \leq L} \left(\prod_{p\in (T_{j-1},T_j]} \left(1 + \frac{|b_{j, T_L, \bm{\a}, \bm{\b}}(p)|^2}{p} + O\left(\frac{1}{p^2}\right)\right) + O(e^{-50\b_\ast^2 K_j}) \right)\\
&\ll  T (\log T)^{\b_1^2 + \cdots + \b_m^2} \prod_{1\leq j < k \leq m} |\zeta(1 + i(\a_j - \a_k) + 1/ \log T )|^{2\b_j \b_k},
\end{align*}
where the latter bound follows from an application of lemma \ref{lem:cosPrimeSum}. This concludes the proof of Proposition \ref{prop:JBound}. \qed

\section{Proof of Proposition \ref{prop:I0Bound}}

Recall 
\begin{align*}
\I_0 &= \int_{\G_m} \prod_{k\leq m} \zeta(\tfrac{1}{2} + i(t+\a_k))\\
&\times \prod_{j\leq L} \exp\left((\b_k-1) \p_{j,T_L}(\tfrac{1}{2}+i(t+\a_k)) + \b_k  \p_{j,T_L}(\tfrac{1}{2}-i(t+\a_k)) \right) w(t/T) \ dt.
\end{align*}
To compute this, we will first need to replace the exponential by a Dirichlet polynomial.
To accomplish this, we will write $\I_0$ as a telescoping sum and control the size of the intermediate increments.
For $0 \leq J \leq L$ denote
\begin{align*}
& \qquad\qquad\qquad \I_0^{(J)} = \int_{\G_m} \prod_{k\leq m} \zeta(\tfrac{1}{2} + i(t+\a_k))\\
\times \prod_{j\leq J} &\n_{j,T_L}(\tfrac{1}{2} + i(t + \a_k); \tfrac{1}{2}(\b_k - 1))^2 \overline{\n_{j,T_L}(\tfrac{1}{2} + i(t + \a_k); \tfrac{1}{2}\b_k)}^2 \\ 
\times \prod_{J < j\leq L} &\exp\left((\b_k-1) \p_{j,T_L}(\tfrac{1}{2}+i(t+\a_k)) + \b_k  \p_{j,T_L}(\tfrac{1}{2}-i(t+\a_k)) \right) w(t/T) \ dt.
\end{align*}
Since $\I_0 = \I_0^{(0)}$, we may decompose
\[
\I_0 = \I_0^{(L)} - \sum_{J \leq L} \left(\I_0^{(J)} - \I_0^{(J-1)}\right)
\]
To prove Proposition \ref{prop:I0Bound} we will give a lower bound for $|I_0^{(L)}|$ and then show that
\[
\sum_{J\leq L} \big|\I_0^{(J)} - \I_0^{(J-1)} \big| \leq \frac{|I_0^{(L)}|}{2}.
\]
To control these differences, we will use the following consequence of lemma \ref{lem:expTaylorSeries} and the definition of $\G_m$.

\begin{lemma}\label{lem:incrementBound}
For $0\leq J \leq L$
\begin{align*}
&\big|\I_0^{(J)} - \I_0^{(J-1)} \big| \\
\ll &e^{-50\b_\ast^2 K_J} \int_{\G_m} \prod_{k\leq m} |\zeta(\tfrac{1}{2} + i(t+\a_k))| \cdot  |\n_{T_L}(\tfrac{1}{2} + i(t + \a_k); \tfrac{1}{2}(\b_k - 1)) \n_{T_L}\tfrac{1}{2} + i(t + \a_k); \tfrac{1}{2}\b_k)|^2  dt.
\end{align*}
\end{lemma}

To prove Proposition \ref{prop:I0Bound}, we will now just need the following two estimates.

\begin{prop}\label{prop:I0Llower}
\[
|\I_{0}^{(L)}| \gg  T \prod_{p\leq T_L} \left(1 + \sum_{1\leq j,k \leq m} \frac{\b_j \b_k}{p^{1 + i(\a_j-\a_k)}} \right).
\]
\end{prop}

\begin{prop}\label{prop:RHtwisted}
Assuming the Riemann hypothesis
\begin{align*}
\int_{\G_m} \prod_{k\leq m} |\zeta(\tfrac{1}{2} + i(t + &\a_k))| \cdot  |\n_{T_L}(\tfrac{1}{2} + i(t + \a_k); \tfrac{1}{2}(\b_k - 1)) \n_{T_L}\tfrac{1}{2} + i(t + \a_k); \tfrac{1}{2}\b_k)|^2  dt \\
&\ll  T \prod_{p\leq T_L} \left(1 + \sum_{1\leq j,k \leq m} \frac{\b_j \b_k}{p^{i(\a_j-\a_k)}} \right).
\end{align*}

\end{prop}
\noindent
Before proving these two statements, let us briefly see how they imply Proposition \ref{prop:I0Bound}.
\begin{proof}[Proof of Proposition  \ref{prop:I0Bound}]

Combining Propositions \ref{prop:RHtwisted} and \ref{prop:I0Llower} with lemma \ref{lem:incrementBound}, we see that
\[
\sum_{J\leq L} \big|\I_0^{(J)} - \I_0^{(J-1)} \big| \ll  |I_0^{(L)}| \sum_{J\leq L} e^{-50\b_\ast^2K_J}.
\]
By definition of $K_J$, the sum on the right hand side is equal to
\[
\sum_{J\leq L} \exp\left(-50\b_\ast^2 \frac{(\log_2 T)^{3/2}}{e^{3J/4}}\right).
\]
Because $T_L \leq T^\d$, it also follows that $L \leq 2 \log_3 T + \log \d$.
Therefore by summing in reverse starting at $J = L$, we see that this sum is bounded by 
\[
\sum_{j\geq 1} \exp\left(-50\b_\ast^2 e^{-3\log\d / 4}  e^{3j/4}\right) \leq \sum_{j\geq 1} \exp\left(-50\b_\ast^2 e^{-3\log\d / 4}    j \right) \ll \exp\left(-50\b_\ast^2 e^{-3\log\d / 4} \right) .
\]
Therefore if we choose $\d > 0$ sufficiently small in terms of $\bm{\b}$ we may ensure that
\[
\sum_{J\leq L} \big|\I_0^{(J)} - \I_0^{(J-1)} \big| \leq \frac{|I_0^{(L)}|}{2}.
\]
Therefore it follows that 
\[
|I_0| \gg   T \prod_{p\leq T_L} \left(1 + \sum_{1\leq j,k \leq m} \frac{\b_j \b_k}{p^{1 + i(\a_j-\a_k)}} \right).
\]

\end{proof}

\subsection{Proof of Proposition \ref{prop:I0Llower}}
To estimate $\I_0^{(L)}$, we will first write $\I_0^{(L)} = \J_1 - \J_2$, where
\[
\J_1 = \int_T^{2T} \prod_{k\leq m} \zeta(\tfrac{1}{2} + i(t+\a_k)) \n_{T_L}(\tfrac{1}{2} + i(t + \a_k); \tfrac{1}{2}(\b_k - 1))^2 \overline{\n_{T_L}(\tfrac{1}{2} + i(t + \a_k); \tfrac{1}{2}\b_k)}^2 w(t/T) \ dt
\]
and 
\[
\J_2 = \int_{[T,2T] \setminus \G_m} \prod_{k\leq m} \zeta(\tfrac{1}{2} + i(t+\a_k)) \n_{T_L}(\tfrac{1}{2} + i(t + \a_k); \tfrac{1}{2}(\b_k - 1))^2 \overline{\n_{T_L}(\tfrac{1}{2} + i(t + \a_k); \tfrac{1}{2}\b_k)}^2 w(t/T) \ dt.
\]
We will evaluate $\J_1$ first. We will make use of the following approximate functional equation.

\begin{lemma}\label{lem:AFEpure} 
Let
\[
V(x,t) = \frac{1}{2\pi i} \int_{(1)} \frac{e^{s^2}}{s} \left(\frac{t^{3m}}{ x}\right)^s ds
\]
and $\tau_{\bm{\a}}(n) = \sum_{n_1\cdots n_m = n} n_1^{-i\a_1} \cdots n_m^{-i\a_m}$.
For $\a_j \leq T/2$ and $t\in [T,2T]$
\[
\prod_{k\leq m} \zeta(\tfrac{1}{2} + i t + i \a_k) = \sum_n \frac{\tau_{\bm{\a}}(n)}{n^{1/2 + i t}} V(n,t) + O(1/T).
\]
\end{lemma}

\begin{proof}
We will evaluate the integral 
\[
I = \frac{1}{2\pi i} \int_{(1)} \frac{e^{s^2}}{s} t^{3ms} \prod_{k\leq m} \zeta(\tfrac{1}{2} + it + i\a_k + s) ds
\]
in two ways.
Expanding $\prod_{k\leq m} \zeta(\tfrac{1}{2} + it + i\a_k + s)$ into its Dirichlet series and simplifying shows
\[
I =  \sum_n \frac{\tau_{\bm{\a}}(n)}{n^{1/2 + i t}} V(n,t).
\]
Alternatively, by shifting the contour to $\Re s = -1$, we pass over poles at $s = 0$ and $s = -\tfrac{1}{2} - it - i\a_k$.
Only the residue at $s=0$ contributes because of the rapid decay of $e^{s^2}$ and because $|t + \a_k| \geq T/2$.  Therefore
\[
I = \prod_{k\leq m} \zeta(\tfrac{1}{2} + i t + i \a_k) + \frac{1}{2\pi i} \int_{(-1)}\frac{e^{s^2}}{s} t^{3ms} \prod_{k\leq m} \zeta(\tfrac{1}{2} + it + i\a_k + s) ds + O_A(T^{-A}).
\]
To conclude, by the standard estimate $\zeta(-\tfrac{1}{2} + i t) \ll 1 + |t|$, we may bound this final integral by
\[
\ll \int_\R \frac{e^{-y^2}}{y + 1} T^{-3m} \prod_{j\leq s} (1 + |t + \a_j + y|) dy.
\]
The integral over the region $|y| \leq T$ can be bounded by
\[
\ll \int_0^T \frac{e^{-y^2}}{y + 1} T^{-3m} \times T^m \ dy \ll T^{-1},
\]
and the integral over the region $|y| \geq T$ is
\[
\ll \int_{T}^\infty \frac{e^{-y^2}}{y} T^{-2m} \times y^m  \ dy \ll T^{-1}.
\]
\end{proof}
\noindent
Next we will need to understand the coefficients of the Dirichlet polynomials
\[
\prod_{k\leq m} \n_{j,T_L}(\tfrac{1}{2} + i(t + \a_k); \tfrac{1}{2}(\b_k - 1))^2 = \sum_{n} \frac{q_j(n)}{n^{1/2 + it}},
\]
\[
\prod_{k\leq m} \n_{T_L}(\tfrac{1}{2} + i(t + \a_k); \tfrac{1}{2}(\b_k - 1))^2 = \sum_{n} \frac{q(n)}{n^{1/2 + it}},
\]
and
\[
\prod_{k\leq m} \n_{j,T_L}(\tfrac{1}{2} + i(t + \a_k); \tfrac{1}{2}\b_k)^2 = \sum_{n} \frac{r_j(n)}{n^{1/2 + it}},
\]
\[
\prod_{k\leq m} \n_{T_L}(\tfrac{1}{2} + i(t + \a_k); \tfrac{1}{2}\b_k)^2 = \sum_{n} \frac{r(n)}{n^{1/2 + it}}.
\]
Notice that each $q_j(n)$ is the twofold Dirichlet convolution of $b_{j,T_L, \bm{\a}, \tfrac{1}{2}(\bm{\b}-1)}(n)$ with itself and each $r_j(n)$ is the twofold Dirichlet convolution of $b_{j,T_L, \bm{\a}, \tfrac{1}{2}\bm{\b}}(n)$ with itself, where $\tfrac{1}{2}(\bm{\b}-1)$ and $\tfrac{1}{2}\bm{\b}$ denote the vectors with $j^{\text{th}}$  element $\tfrac{1}{2}(\b_j - 1)$ and $\tfrac{1}{2}\b_j$ respectively.
As before, we will define two more sets of coefficients:
Let $q'_j(n)$ be the twofold Dirichlet convolution of $b'_{j,T_L, \bm{\a}, \tfrac{1}{2}(\bm{\b}-1)}(n)$ with itself and let $r'_j(n)$ be the twofold Dirichlet convolution of $b'_{j,T_L, \bm{\a}, \tfrac{1}{2}\bm{\b}}(n)$ with itself.
Finally define $q''_j(n)$ and $r''_j(n)$ in an analogous fashion.
As before, the key point is that $q_j'$ and $r'_j$ are multiplicative approximations of $q_j$ and $r_j$ and that $q''_j$ and $r''_j$ are non-negative multiplicative coefficients satisfying $|q_j|,|q'_j| \leq q''$ and $|r_j|,|r'_j| \leq r''$.
Before proceeding further, we will need the following estimate.

\begin{lemma}\label{lem:prodCoefPatrol}
For $p \in (T_{j-1},T_j]$
\[
q_j(p^l), r_j(p^l) \ll m^2 \b_{\ast}^l l^{4m} e^{-l \log (l/2m)/4m + 2l}.
\]
\end{lemma}
\begin{proof}
The bound for $q_j$ follows by combining the formula
\[
q_j(p^l) = \sum_{x + y = l} b_{j,T_L, \bm{\a}, \tfrac{1}{2}(\bm{\b}-1)}(p^x) b_{j,T_L, \bm{\a}, \tfrac{1}{2}(\bm{\b}-1)}(p^y)
\]
with the bound for $b_{j,T_L, \bm{\a}, \tfrac{1}{2}(\bm{\b}-1)}(p^x)$ given in lemma \ref{lem:nCoeffPatrol} and noting that either $x \geq l/2$ or $y \geq l/2$.
The argument for $r_j$ is the same.
\end{proof}

We can now compute
\begin{align*}
\J_1 &=  \frac{1}{2\pi i}  \sum_{n} \sum_{h,k < T^{1/2}} \frac{\tau_{\bm{\a}}(n) q(h)\overline{r(k)}}{n^{1/2} h^{1/2} k^{1/2}} \int_{(1)} \frac{e^{s^2}}{s} \int_{T}^{2T}  \left( \frac{nh}{k} \right)^{-it} \frac{t^{3ms}}{n^s} w(t/T) \ dt \ ds \\
&+ O\left( \frac{1}{T} \int_T^{2T} \prod_{k\leq m} |\n_{T_L}(\tfrac{1}{2} + i(t + \a_k); \tfrac{1}{2}(\b_k - 1)) \n_{T_L}(\tfrac{1}{2} + i(t + \a_k); \tfrac{1}{2}\b_k)|^2  dt \right).
\end{align*}
Because 
\[
\prod_{k\leq m} \n_{T_L}(\tfrac{1}{2} + i(t + \a_k); \tfrac{1}{2}(\b_k - 1)) \n_{T_L}(\tfrac{1}{2} + i(t + \a_k); \tfrac{1}{2}\b_k)
\]
is a short Dirichlet polynomial whose coefficients $c(n)$ satisfy $c(n) \ll_{\ve} n^\ve$, it follows that the error term is $\ll_\ve T^{\ve}$ by lemma \ref{lem:MVDP}.
Returning now to the main term, the next step is to discard the terms where $nh \neq k$.
An exercise in contour integration shows that
\[
t^{j} \frac{\del^j}{\del t^j} V(x,t) \ll_{A,j} \left(1 + |x/t^{3m}|\right)^{-A}. 
\]
Therefore repeated integration by parts implies that 
\[
\int_{T}^{2T} w(t/T) \left(\frac{n h}{k}\right)^{-it} V(x,t) \ dt \ll_{j,A} \frac{(1 + n/T^{3m})^{-A}}{|\log(n h/ k)|^j T^j} .
\]
Because $h$ and $k$ are at most $T^{1/2}$, it follows that $\log(nh/k) \gg T^{-1/2}$ if $nh \neq k$.
Therefore the contribution to the sum of terms with $hn \neq k$ is $O_A(T^{-A})$ because $h,k < T^{1/2}$.

Now by shifting the contour to $\Re s = -1/4$ and simplifying the diagonal terms, we find that 
\[
\J_1 = T\|w\|_1 \sum_{\substack{h,k \leq T^{1/2} \\ h|k}} \frac{\tau_{\bm{\a}}(k/h) q(h) \overline{r(k)}}{k} + O(T^{1- \ve}).
\]
By multiplicativity, we may factor the inner sum as
\[
\sum_{\substack{h,k \leq T^{1/2} \\ h|k}} \frac{\tau_{\bm{\a}}(k/h) q(h) \overline{r(k)}}{k} = \prod_{j\leq L} \sum_{\substack{p| h,k \Rightarrow p\in (T_{j-1},T_j] \\ h|k }} \frac{\tau_{\bm{\a}}(k/h) q_j(h) \overline{r_j(k)}}{k}.
\]
Next notice if $q_j(h) \neq q'_j(h)$, then it must be that $\Omega(h) \geq 100\b_\ast^2 K_j$.
So replacing $q_j$ and $r_j$ with $q'_j$ and $r'_j$ respectively incurs an error of order
\begin{align*}
&e^{-100\b_\ast^2 K_j} \sum_{\substack{p| h,k \Rightarrow p\in (T_{j-1},T_j] \\ h|k }} \frac{e^{\Omega(h)}|\tau_{\bm{\a}}(k/h)| q''_j(h) r''_j(k)}{k} \\
&\ll e^{-100\b_\ast^2 K_j} \prod_{p\in(T_{j-1},T_j]} \left(1 + \frac{2 e \b_\ast^2}{p} + O \left(\frac{1}{p^2}\right)\right) \ll  e^{-50 \b_\ast^2 K_j}.
\end{align*}
for each $j$.
Here we have used the divisor bound $|\tau_{\bm{\a}}(p^y)| \ll_\ve p^{\ve y}$ and lemma \ref{lem:prodCoefPatrol} to bound the contribution of the terms of order smaller than $1/p^2$.
Therefore, suppressing terms of order $T^{1-\ve}$, we see that
\begin{align*}
\J_1 &= T\|w\|_1 \prod_{j\leq L}\left( \prod_{p\in(T_{j-1},T_j]}\left( \sum_{0 \leq x \leq y} \frac{\tau_{\bm{\a}}(p^{y-x}) q_j(p^x) \overline{r_j(p^y)}}{p^y} \right) + O (e^{-50 \b_\ast^2 K_j})\right)\\
&= T\|w\|_1 \prod_{j\leq L}\left( \prod_{p\in(T_{j-1},T_j]}\left( 1 + \frac{q_j(p) \overline{r_j(p)} + \tau_{\bm{\a}}(p) \overline{r_j(p)} }{p} + O \left(\frac{1}{p^2}\right)\right) + O (e^{-50 \b_\ast^2 K_j})\right).
\end{align*}
A quick computation shows that 
\[
q_j(p) + \tau_{\bm_{\a}}(p) = a_{T_L}(p) \sum_{k\leq m} \b_k p^{-i \a_k} = r_j(p).
\]
Therefore, recalling  $a_{T_L}(p) = p^{-1/\log T_L} (1 - \log p/ \log T_L)$, we deduce
\begin{align*}
\J_1 = T\|w\|_1 \prod_{j\leq L}\left( \prod_{p\in(T_{j-1},T_j]}\left( 1 + \frac{1}{p} \Bigg|\sum_{k \leq m} \b_k p^{-i\a_k}\Bigg|^2 + O \left( \frac{\log p}{p \log T_{L}} + \frac{1}{p^2}\right)\right) + O (e^{-50 \b_\ast^2 K_j})\right),
\end{align*}
and we may readily conclude
\begin{prop}\label{prop:J1Lower}
\[
\J_1 \gg  T \prod_{p \leq T_{L}} \left(1 + \sum_{1\leq j , k \leq m} \frac{\b_j \b_k}{p^{1+ i(\a_j - \a_k)}}\right).
\]
\end{prop}

We will now show, for $\d$ sufficiently small, that $|\J_2| \leq |\J_1|/2$.
To do this we will need to decompose the set $[T,2T] \setminus \G_m$ into many pieces.
First recall the definition
\[
\G := \left\{t\in[T/2,5T/2]: |P_{j,T_L} (\tfrac{1}{2} + i t)| \leq K_j \text{ for all } 1\leq j\leq L\right\}.
\]
Now given a subset $A$ of $[m] := \{1,\ldots, m\}$, we will define 
\[
\G_A :=  \left\{t\in [T,2T]: t + \a_k \in \G \text{ if and only if } k \in A\right\},
\]
and for each $1 \leq j \leq L$ let
\begin{align*}
\B_j := \{t\in [T/2&,5T/2]: |\p_{r,T_s} (\tfrac{1}{2} + i t)| \leq K_j \text{ for all } 1\leq r < j\text{ and } r\leq s \leq  L \\
&\text{ but } |\p_{j,T_s} (\tfrac{1}{2} + i t)|  > K_{j} \text{ for some } j \leq s \leq L \}.
\end{align*}
To evaluate $\J_2$, we must evaluate an integral over all of the $\G_A$ with $A$ ranging over all proper subsets of $[m]$.
Without loss of generality we will write $A = [m]\setminus [a]$.
For each $t \in \G_A$, there is a function $F_t: [a] \rightarrow [L]$ such that $t + \a_j \in \B_{f(j)}$.
We will further partition $\G_A$ into the sets
\[
\B_{A,n} = \{t\in \G_A: \min_{j \in [a]} F_t(j) = n \}.
\]
With this new notation, we may write
\[
[T,2T] \setminus \G_m =  \bigsqcup_{n\leq L} \bigsqcup_{A \subsetneq [m]} \B_{A,n}.
\]

When $n > 1$, we may proceed in a similar manner to \cite{CurranCorrelations}.
Using Lemma \ref{lem:logZetaUpperBound}  with $X = T_{n-1}$  and the definition of $\B_{A,n}$ we find
\begin{align*}
\int_{\B_{A,n}}  &\prod_{k\leq m} |\zeta(\tfrac{1}{2} + i(t+\a_k))| |\n_{T_L}(\tfrac{1}{2} + i(t + \a_k); \tfrac{1}{2}(\b_k - 1))\n_{T_L}(\tfrac{1}{2} + i(t + \a_k); \tfrac{1}{2}\b_k)|^2 w(t/T) \ dt \\
&\ll
\int_{\B_{A,n}} \prod_{k = 1}^m \exp\Bigg(\Re \Bigg( \sum_{j < n}  \p_{j,T_{n-1}} (\tfrac{1}{2}  + i(t + \a_k)) + 1/ c_{n - 1} \Bigg) \\
&\qquad\qquad\qquad\qquad\times |\n_{T_L}(\tfrac{1}{2} + i(t + \a_k); \tfrac{1}{2}(\b_k - 1))\n_{T_L}(\tfrac{1}{2} + i(t + \a_k); \tfrac{1}{2}\b_k)|^2  \ dt \\
&\ll  ~e^{1/ c_{n - 1}}  \int_{\B_{A,n}}\prod_{k = 1}^m \prod_{j < n} \exp\left(\Re \p_{j,T_{n-1}} (\tfrac{1}{2}  + i(t + \a_k)) \right) \\
&\qquad\qquad\qquad\qquad\times |\n_{T_L}(\tfrac{1}{2} + i(t + \a_k); \tfrac{1}{2}(\b_k - 1))\n_{T_L}(\tfrac{1}{2} + i(t + \a_k); \tfrac{1}{2}\b_k)|^2  \ dt \\
&\ll  ~e^{1/ c_{n - 1}} \max_{\substack{\ell \in [a] \\ s \in [L]}}
\int_T^{2T} |\p_{n,T_s}(\tfrac{1}{2} + i (t + \a_\ell))/K_{n}|^{2 \lceil 1/ 20 c_{n}\rceil} \prod_{k = 1}^m \prod_{j < n} |\n_{j,T_{n-1}} (\tfrac{1}{2}+ i(t + \a_k) ; \tfrac{1}{2}) |^2  \\
&\qquad\qquad\qquad\qquad\times |\n_{T_L}(\tfrac{1}{2} + i(t + \a_k); \tfrac{1}{2}(\b_k - 1))\n_{T_L}(\tfrac{1}{2} + i(t + \a_k); \tfrac{1}{2}\b_k)|^2  \ dt.
\end{align*}
So by lemma \ref{lem:Splitting} we need the following mean value calculations.

\begin{prop}\label{prop:RHMVDP}
For $j < n$
\begin{align*}
\int_{T}^{2T} \prod_{k = 1}^m &|\n_{j,T_{n-1}} (\tfrac{1}{2}+ i(t + \a_k) ; \tfrac{1}{2}) \n_{j,T_L}(\tfrac{1}{2} + i(t + \a_k); \tfrac{1}{2}(\b_k - 1))\n_{j,T_L}(\tfrac{1}{2} + i(t + \a_k); \tfrac{1}{2}\b_k)|^2 dt\\
\leq \ & T \prod_{p \in (T_{j-1},T_j]}  \left(1 + \frac{1}{p} \sum_{1\leq j , k \leq m} \frac{\b_j \b_k}{p^{i(\a_j - \a_k)}}  + O  \left(\frac{\log p}{p \log T_{n-1}} + \frac{1}{p^2}\right) \right) + O(e^{-50\b_\ast^2 K_j}).
\end{align*}
\end{prop}

\begin{proof}
All of the following calculations are similar to earlier computations, so we will just sketch their proofs.
The coefficients of the Dirichlet polynomial 
\[
\prod_{k = 1}^m\n_{j,T_{n-1}} (\tfrac{1}{2}+ i(t + \a_k) ; \tfrac{1}{2}) \n_{j,T_L}(\tfrac{1}{2} + i(t + \a_k); \tfrac{1}{2}(\b_k - 1))\n_{j,T_L}(\tfrac{1}{2} + i(t + \a_k); \tfrac{1}{2}\b_k)
\]
are given by the triple convolution 
\[
a(n) := b_{j,T_{n-1}, \bm{\a}, \bm{1/2}}\ast b_{j,T_{L}, \bm{\a}, \tfrac{1}{2}(\bm{\b} - 1)} \ast b_{j,T_{L}, \bm{\a}, \tfrac{1}{2}\bm{\b}} (n),
\]
where $ \bm{1/2}$ is a vector of $m$ copies of $1/2$. 
Using Rankin's trick, one can replace these with the multiplicative coefficients
\[
a'(n) := b'_{j,T_{n-1}, \bm{\a}, \bm{1/2}}\ast b'_{j,T_{L}, \bm{\a}, \tfrac{1}{2}(\bm{\b} - 1)} \ast b'_{j,T_{L}, \bm{\a}, \tfrac{1}{2}\bm{\b}} (n)
\]
at a cost of $O(e^{-50\b_\ast^2 K_j})$.
Using lemma \ref{lem:nCoeffPatrol}, one may then show that 
\[
|a'(p)|^2 = \sum_{1\leq j , k \leq m} \frac{\b_j \b_k}{p^{i(\a_j - \a_k)}}  + O  \left(\frac{\log p}{\log T_{n-1}}\right) 
\]
and
\[
\sum_{r \geq 2} \frac{|a(p^r)|^2}{p^r} \ll \frac{1}{p^2}.
\]
The claim now follows by multiplicativity and  lemma \ref{lem:MVDP}.
\end{proof}

\begin{prop}
For $1 < n \leq L$
\begin{align*}
\int_{T}^{2T} |\p_{n,T_s}(\tfrac{1}{2} + i (t + \a_\ell))/K_{n}|^{2 \lceil 1/ 20 c_{n}\rceil} & \prod_{k = 1}^m |\n_{n,T_L}(\tfrac{1}{2} + i(t + \a_k); \tfrac{1}{2}(\b_k - 1))\n_{n,T_L}(\tfrac{1}{2} + i(t + \a_k); \tfrac{1}{2}\b_k)|^2 \ dt\\
&\ll T e^{-\log(1/c_n)/40c_n} .
\end{align*}
\end{prop}

\begin{proof}
By Cauchy Schwarz, the relevant  mean value is at most
\[
\left(\int_{T}^{2T} |\p_{n,T_s}(\tfrac{1}{2} + i (t + \a_\ell))/K_{n}|^{2 \lceil 1/ 20 c_{n}\rceil} \ dt \right)^{1/2 }
\]
\[
\times \left(\int_{T}^{2T}  \prod_{k = 1}^m |\n_{T_L}(\tfrac{1}{2} + i(t + \a_k); \tfrac{1}{2}(\b_k - 1))\n_{T_L}(\tfrac{1}{2} + i(t + \a_k); \tfrac{1}{2}\b_k)|^2 \ dt \right)^{1/2}.
\]
By Proposition 4.2 of \cite{CurranCorrelations}, the first integral is $\ll  T e^{-\log(1/c_n)/20c_n}$.
Using the same reasoning in the previous proof, one may show that the second integral is 
\[
\ll \prod_{p\in {(T_{n-1},T_n]}} \left(1 + O \left(\frac{1}{p}\right)\right) \ll 1
\]
because $\log T_n/\log T_{n-1} = e$ for $n > 1$. The claim now follows.

\end{proof}

\begin{prop}
For $j > n$
\begin{align*}
\int_{T}^{2T} \prod_{k = 1}^m |\n_{j,T_L}&(\tfrac{1}{2} + i(t + \a_k); \tfrac{1}{2}(\b_k - 1))\n_{j,T_L}(\tfrac{1}{2} + i(t + \a_k); \tfrac{1}{2}\b_k)|^2 dt \\
&\leq T \prod_{p \in (T_{j-1},T_j]}  \left(1 + O \left(\frac{1}{p}\right) \right) + O(e^{-50\b_\ast^2 K_j}).
\end{align*} 
\end{prop}

\begin{proof}
The proof is a simpler version of the proof of Proposition \ref{prop:RHMVDP}. 
One now has a double convolution instead of a triple convolution, and  uses a cruder bound for the Dirichlet polynomial coefficients. The details are omitted.
\end{proof}

The previous three propositions and Mertens' estimate now imply that
\begin{align*}
&\int_{\B_{A,n}}  \prod_{k\leq m} |\zeta(\tfrac{1}{2} + i(t+\a_k))| |\n_{T_L}(\tfrac{1}{2} + i(t + \a_k); \tfrac{1}{2}(\b_k - 1))\n_{T_L}(\tfrac{1}{2} + i(t + \a_k); \tfrac{1}{2}\b_k)|^2 w(t/T) \ dt \\
&\ll T \prod_{j < n} \left(\prod_{p \in (T_{j-1},T_j]}  \left(1 + \frac{1}{p} \sum_{1\leq j , k \leq m} \frac{\b_j \b_k}{p^{i(\a_j - \a_k)}}  + O  \left(\frac{\log p}{p \log T_{n-1}} + \frac{1}{p^2}\right) \right) + O(e^{-50\b_\ast^2 K_j}) \right)\\
\end{align*}
\begin{align*}
&\qquad\qquad\times e^{1/c_{n-1} - \log(1/c_n)/40c_n} \prod_{n < j \leq L} \left(\prod_{p \in (T_{j-1},T_j]}  \left(1 + O \left(\frac{1}{p}\right) \right) + O(e^{-50\b_\ast^2 K_j})\right) \\
&\ll T \exp \left(1/c_{n-1} - \log(1/c_n)/40c_n + O\left((\log_2 T)^2 e^{-n} \right) \right) \prod_{p\leq T_{n-1}} \left(1 + \frac{1}{p} \sum_{1\leq j , k \leq m} \frac{\b_j \b_k}{p^{i(\a_j - \a_k)}}\right)\\
&\ll T \exp \left(1/c_{n-1} - \log(1/c_n)/40c_n + O\left((\log_2 T)^2 e^{-n} + L - n\right) \right) \prod_{p\leq T_L} \left(1 + \frac{1}{p} \sum_{1\leq j , k \leq m} \frac{\b_j \b_k}{p^{i(\a_j - \a_k)}}\right).
\end{align*}
\noindent
When $n = 1$, we will instead use the estimate
\begin{prop}\label{prop:badSetInt}
Assuming the Riemann hypothesis, for any $A \subsetneq [m]$
\begin{align*}
\int_{\B_{A,1}} \prod_{k\leq m} \zeta(\tfrac{1}{2} + i(t+\a_k)) \n_{T_L}(\tfrac{1}{2} + i(t &+ \a_k); \tfrac{1}{2}(\b_k - 1))^2 \overline{\n_{T_L}(\tfrac{1}{2} + i(t + \a_k); \tfrac{1}{2}\b_k)}^2 w(t/T) \ dt \\
&\ll_{A} T(\log T)^{-A}.
\end{align*}
\end{prop}

\begin{proof}
By Hölder's inequality, the integral over $\B_{A,1}$ is at most
\[
(\text{meas } \B_{A,1})^{1/3} \left(\int_T^{2T} \prod_{k\leq m} |\zeta(\tfrac{1}{2} + i(t+\a_k))|^3 \ dt\right)^{1/3}
\]
\[
\times \left(\int_T^{2T} |\n_{T_L}(\tfrac{1}{2} + i(t + \a_k); \tfrac{1}{2}(\b_k - 1)) \n_{T_L}(\tfrac{1}{2} + i(t + \a_k); \tfrac{1}{2}\b_k) |^6   \ dt \right)^{1/3}
\]
By the same reasoning used in the proof of Proposition \ref{prop:RHMVDP}, one may show the final integral is $\ll T (\log T)^{O(1)}$.
By the main theorem in \cite{CurranCorrelations}, the first integral is also  $\ll T (\log T)^{O(1)}$.
The claim now follows by using lemma 2.7 of \cite{CurranCorrelations} to bound $\text{meas } \B_{A,1}.$
\end{proof}

We now have all the necessary tools to show
\begin{prop}\label{prop:J2J1comp}
Assuming the Riemann hypothesis, if $\d$ is sufficiently small in terms of $\bm{\b}$
\[
|\J_2| \leq |\J_1|/2.
\]
\end{prop}

\begin{proof}
The preceding calculations in tandem with Proposition \ref{prop:J1Lower} imply that for $n > 1$
\begin{align*}
\int_{\B_{A,n}}  \prod_{k\leq m} |\zeta(\tfrac{1}{2} &+ i(t+\a_k))| |\n_{T_L}(\tfrac{1}{2} + i(t + \a_k); \tfrac{1}{2}(\b_k - 1))\n_{T_L}(\tfrac{1}{2} + i(t + \a_k); \tfrac{1}{2}\b_k)|^2 w(t/T) \ dt\\
&\ll \exp\left(1/c_{n-1} - \log(1/c_n)/40c_n + O\left((\log_2 T)^2 e^{-n} + L - n\right)\right) |\J_1|.
\end{align*}
Therefore by summing over all $A \subsetneq [m]$ and applying Proposition \ref{prop:badSetInt} we see that
\begin{equation}\label{eq:J12comp}
|\J_2| \ll |\J_1| \sum_{2\leq n \leq L}\exp\left(1/c_{n-1} - \log(1/c_n)/40c_n + O\left((\log_2 T)^2 e^{-n} + L - n\right)\right) + O_A\left(T(\log T)^{-A}\right).
\end{equation}
The $T(\log T)^{-A}$ term is  negligible. 
To bound the sum over $n$, we use the definition of the $c_n$ to find
\begin{align*}
\sum_{2\leq n \leq L}\exp\left(1/c_{n-1} - \log(1/c_n)/40c_n + O\left((\log_2 T)^2 e^{-n} + L - n\right)\right)\\
= \sum_{2\leq n \leq L}\exp\left(e^{-n} (\log_2 T)^2 (O(1) + \tfrac{1}{40} n - \tfrac{1}{20} \log_3 T)\right) + O(L - n) ).
\end{align*}
Because $T_L \leq T^{\d}$ it follows that $L \leq 2 \log_3 T + \log \d$ so the sum is at most
\begin{align*}
\sum_{2\leq n \leq L} \exp\left(e^{-n} (\log_2 T)^2 (O(1) + \tfrac{1}{40}\log \d)\right) + O(L - n) ).
\end{align*}
By summing in reverse, we may bound this sum by
\[
\sum_{j \geq 0} \exp\left(e^j (O(1) + \tfrac{1}{40}\log \d) + O(j) \right)\leq  \sum_{j \geq 1} \exp\left((O(1) + \tfrac{1}{40}\log \d ) j \right) \leq \exp\left(O(1) + \tfrac{1}{40}\log \d  \right).
\]
Therefore by taking $\d$ sufficiently small in terms of $\bm{\b}$, we can ensure that this sum times the implicit constant in (\ref{eq:J12comp}) is at most $1/3$, so the claim follows.
\end{proof}
\noindent
Combining this with Proposition \ref{prop:J1Lower} completes the proof of Proposition \ref{prop:I0Llower} \qed

\subsection{Proof of Proposition \ref{prop:RHtwisted}}

Now only Proposition \ref{prop:RHtwisted} is needed to complete the proof of Proposition \ref{prop:I0Bound}. 
To accomplish this, we will now use Lemma \ref{lem:logZetaUpperBound} with $X = T_{L}$  and the definition of $\G_m$ to deduce

\begin{align*}
\int_{\G_m}  &\prod_{k\leq m} |\zeta(\tfrac{1}{2} + i(t+\a_k))| |\n_{T_L}(\tfrac{1}{2} + i(t + \a_k); \tfrac{1}{2}(\b_k - 1))\n_{T_L}(\tfrac{1}{2} + i(t + \a_k); \tfrac{1}{2}\b_k)|^2 w(t/T) \ dt \\
&\ll    \int_{\G_m}\prod_{k = 1}^m \prod_{j \leq L} \exp\left(\Re \p_{j,T_{L}} (\tfrac{1}{2}  + i(t + \a_k)) \right) \\
&\qquad\qquad\qquad\times |\n_{T_L}(\tfrac{1}{2} + i(t + \a_k); \tfrac{1}{2}(\b_k - 1))\n_{T_L}(\tfrac{1}{2} + i(t + \a_k); \tfrac{1}{2}\b_k)|^2  \ dt \\
&\ll 
\int_T^{2T} \prod_{k = 1}^m  |\n_{T_{L}} (\tfrac{1}{2}+ i(t + \a_k) ; \tfrac{1}{2}) |^2  \\
&\qquad\qquad\qquad\times |\n_{T_L}(\tfrac{1}{2} + i(t + \a_k); \tfrac{1}{2}(\b_k - 1))\n_{T_L}(\tfrac{1}{2} + i(t + \a_k); \tfrac{1}{2}\b_k)|^2  \ dt.
\end{align*}
So all that remains to is compute the mean square of the Dirichlet polynomial 
\[
\prod_{k = 1}^m \n_{T_{L}} (\tfrac{1}{2}+ i(t + \a_k) ; \tfrac{1}{2}) \n_{T_L}(\tfrac{1}{2} + i(t + \a_k); \tfrac{1}{2}(\b_k - 1))\n_{T_L}(\tfrac{1}{2} + i(t + \a_k); \tfrac{1}{2}\b_k).
\]
Proposition \ref{prop:RHMVDP} with $n = L + 1$ and lemma \ref{lem:Splitting} imply that this mean value is 
\[
\ll T \prod_{p \leq T_L} \left(1 + \frac{1}{p} \sum_{1\leq j,k \leq m} \frac{\b_j \b_k}{p^{i(\a_j - \a_k)}}\right).
\]
This now concludes the proof of Proposition \ref{prop:RHtwisted}, and therefore also the proof of Proposition \ref{prop:I0Bound}.

\qed

\section{Proof of Proposition \ref{prop:IkBound}}
All that remains is to bound the quantities
\begin{align*}
\I_k = \int_{\G_m} |\zeta(\tfrac{1}{2} + i (t + \a_k))|^{2m} \prod_{j\leq L} &\exp\left(2(\b_k-m)\Re \p_{j,T_L}(\tfrac{1}{2}+i(t+\a_k)) \right) \\
&\times \prod_{\substack{\ell \leq m \\ \ell \neq k}} \prod_{j\leq L} \exp\left(2\b_\ell\Re \p_{j,T_L}(\tfrac{1}{2}+i(t+\a_\ell)) \right) w(t/T) dt.
\end{align*}
This is very similar to the proof of Proposition \ref{prop:RHtwisted}.
By applying lemma \ref{lem:logZetaUpperBound} with $X = T_L$ and lemma \ref{lem:expTaylorSeries} it follows that
\begin{align*}
\I_k \ll \int_{T}^{2T} |\n_{T_L}(\tfrac{1}{2} &+ i(t + \a_k);m) \n_{T_L}(\tfrac{1}{2} + i(t + \a_k);\b_k - m)|^2 \\
&\times \prod_{\substack{\ell \leq m \\ \ell \neq k}}  |\n_{T_L}(\tfrac{1}{2} + i(t + \a_\ell);\b_\ell)|^2  \ dt.
\end{align*}
Before proceeding it may be helpful to review the definitions made preceding Proposition \ref{prop:Njcoeffs}.
When $j > 1$, the coefficients $a_j(n)$ of the Dirichlet polynomial 
\[
\n_{j,T_L}(\tfrac{1}{2} + i(t + \a_k);m) \n_{j,T_L}(\tfrac{1}{2} + i(t + \a_k);\b_k - m) \prod_{\substack{\ell \leq m \\ \ell \neq k}}  \n_{j,T_L}(\tfrac{1}{2} + i(t + \a_\ell);\b_\ell)
\]
are given by the $m + 1$ fold Dirichlet convolution of $g_{T_L}(n; m) n^{-i \a_k} c_j(n)$ and $g_{T_L}(n; \b_k - m) n^{-i \a_k} c_j(n)$ with $g_{T_L}(n; \b_\ell) n^{-i \a_\ell} c_j(n)$ for all $\ell \leq m$  not equal to $k$.
When $j = 1$ a similar formula holds with $f_{T_L}$ in place of $g_{T_L}$.
As before, one may replace $c_j(n)$ with $1_{p|n \Rightarrow p\in (T_{j-1},T_j]}$ at a cost of $O(e^{-50\b_\ast^2K_j})$ to obtain multiplicative coefficients $a_j'(n)$ which
\[
|a'_j(p)|^2 =  \sum_{1\leq j,k \leq m} \frac{\b_j \b_k}{p^{i(\a_j - \a_k)}} + O\left(\frac{\log p}{\log T_L}\right)
\]
and 
\[
\sum_{r \geq 2} \frac{|a('p^r)|^2}{p^r} \ll \frac{1}{p}
\]
for $T_{j-1} < p \leq T_j$.
As before, this allows us to conclude that
\[
\I_k \ll T \prod_{p\leq T_L} \left(1 + \sum_{1\leq j,k \leq m} \frac{\b_j \b_k}{p^{i(\a_j - \a_k)}} \frac{1}{p}\right).
\]
Proposition \ref{prop:IkBound}, and therefore Theorem \ref{thm:main}, readily follows.

\end{document}